\documentclass[12pt,reqno,a4paper]{amsart}

\usepackage{amssymb,amsthm}
\usepackage{amsfonts,cmtiup,comment,stmaryrd}
\usepackage{hyperref}
\usepackage{mathrsfs,cmtiup}
\usepackage{xcolor}

\makeatletter
\def\blfootnote{\xdef\@thefnmark{}\@footnotetext}
\makeatother

\newtheorem{theorem}{Theorem}[section]
\newtheorem*{theorem*}{Main Theorem}
\newtheorem{lemma}[theorem]{Lemma}
\newtheorem{proposition}[theorem]{Proposition}
\newtheorem{corollary}[theorem]{Corollary}
\newtheorem{question}[theorem]{Question}
\newtheorem{conjecture}[theorem]{Conjecture}

\theoremstyle{definition}
\newtheorem{definition}[theorem]{Definition}
\newtheorem{remark}[theorem]{Remark}

\newtheorem{example}[theorem]{Example}

\newtheorem*{definition*}{Definition}

\newcommand{\Z}{\mathbb Z}
\newcommand{\Q}{\mathbb Q}

\newcommand{\f}{\varphi}
\newcommand{\x}{\mathcal{X}}

\renewcommand{\geq}{\geqslant}
\renewcommand{\leq}{\leqslant}

\newcommand{\ed} {\end{document}}

\let\leq=\leqslant
\let\geq=\geqslant
\numberwithin{equation}{section}

\begin{document}
\title{Finite groups admitting a coprime automorphism satisfying an additional polynomial identity}

\author{W. A. Moens}
\address{Faculty of Mathematics, University of Vienna, Austria}
\email{Wolfgang.Moens@univie.ac.at}

\date{\today}

\subjclass[2010]{}

\begin{abstract} 
	It is known that a finite group with an automorphism $\f$ of coprime order has a soluble radical of $(|\f|,|C_G(\f)|)$-bounded Fitting height and index. We extend this classic result as follows. Let $f(x) = a_0 + a_1 \cdot x + \cdots + a_d \cdot x^d \in \Z[x]$ be a primitive polynomial and let $G$ be a finite group with an automorphism $\f$ of coprime order satisfying $ g^{a_0} \cdot \f(g)^{a_1} \cdots \f^d(g)^{a_d} = 1 $, for all $g \in G$. 
	Then the soluble radical of $G$ has $(d,|C_G(\f)|)$-boundex Fitting height and index. The bounds are made explicit and are particularly good for small values of the degree $d$.
\end{abstract}

\maketitle


\section{Introduction}

We study the structure of finite groups $G$ with an automorphism $\f$ satisfying a given ordered identity $f(x)$. We recall that the polynomial $a_0 + a_1 \cdot x + a_2 \cdot x^2 + \cdots  + a_d \cdot x^d \in \Z[x]$ is said to be an ordered identity of $\f$ if every element $g$ in $G$ satisfies $$g^{a_0} \cdot \f(g)^{a_1} \cdot \f^2(g)^{a_2} \cdots \f^d(g)^{a_d} = 1.$$ 

The most obvious example of such an identity can be obtained as follows. Let $d := |\f|$ be the order of the automorphism. Then the polynomial $f(x) := -1 + x^d$ is clearly an ordered identity of $\f$, and much can be said about the structure of $G$ in terms of $d$ and $m := |C_G(\f)|$, the number of elements fixed by $\f$. For the sake of exposition, we make the additional assumption that the order of the automorphism is coprime to the order of the group. Then the automorphism is said to be \emph{coprime} and $G$ has a large soluble subgroup that can be obtained by few extensions of nilpotent groups. More precisely, the soluble radical of $G$ has $(d,m)$-bounded index in $G$ and $(d,m)$-bounded Fitting height. Such a bound on the index can be found in Hartley's generalization \cite{har} of the famous Brauer---Fowler theorem \cite{brau-fowl}, and a bound on the Fitting height appears in the early work of Thompson \cite{tho64}. We note that Thompson's bound, and later improvements on that bound by Kurzweil \cite{kur} and Turull \cite{tur0}, depend only on $k(d)$ (the number of prime divisors of $d$, counted with multiplicity) and on the Fitting height of $C_G(\f)$ (which is naturally bounded by $m$). Other relevant results are due to Brauer---Fong \cite{bra-fon}, Hartley---Meixner \cite{har-mei}, Hartley---Turau \cite{har-tur}, Pettet \cite{pet}, and Hartley---Isaacs \cite{har-isa}. Cf. Turull's survey \cite{tur94}. \\

Let us further specialize $m$ to $1$. Then the automorphism $\f$ is said to be \emph{fixed-point-free} and the classification of the finite simple groups implies that the group is soluble \cite{row}. In this case, the work of Schult \cite{shu},  Gross \cite{gro}, and Berger \cite{ber} gives the sharp upper bound $k(d)$ on the Fitting height of $G$. By further specializing $d$ to a prime, we force the group to be nilpotent. This is Thompson's celebrated solution \cite{tho59} of the Frobenius conjecture. Higman \cite{hig} showed that the nilpotency class of $G$ is then $d$-bounded, but he did not make the bound explicit. Kreknin and Kostrikin \cite{krek-kos} later found the explicit bound $d^{2^{d}}$. \\

All of these results for $f(x) = -1 + x^d$ and $m = 1$ have recently been generalized to non-zero polynomials $f(x) := a_0 + a_1 \cdot x + \cdots + a_d \cdot x^d \in \Z[x]$ satisfying $\gcd(a_0,a_1,\ldots,a_d) = 1$. Such a polynomial is said to be \emph{primitive}. In fact, consider any finite group $G$ with any fixed-point-free automorphisn $\f$ satisfying the primitive ordered identity $f(x)$. Then the Fitting height of $G$ is at most $2 + 112 \cdot d^2$. Moreover, there exist finitely-many primes $p_1,\ldots,p_l$, dependig \emph{only} on $f(x)$, with the following property. If $\gcd(|G|,p_1 \cdots p_l) = 1$, then the bound on the Fitting height of $G$ can be improved to $4 + (1+c)^2$, where $c$ is the number of irreducible factors of $f(x)$. These two theorems of \cite{khu-moe} depend on the deep results of Hall---Higman \cite{ha-hi}, Shult-Gross-Berger \cite{shu,gro,ber}, and Dade-Jabara \cite{dad,jab}. The upper bound can still be improved to the linear bound $c$ for almost-all polynomials $f(x)$ \cite{moe2}. If $f(x)$ is irreducible, then $G$ is nilpotent of $d$-bounded class \cite{moe0} at most $d^{2^d}$ \cite{moe1}.  \\

These recent generalizations required $m$ to be $1$ but they did not require the automorphism to be coprime. In contrast, we now consider coprime automorphisms but we do not require $m$ to be $1$. The results of Hartley \cite{har} and Thompson-Kurzweil-Turull \cite{tho64,kur,tur0} will then be generalized as follows.

\begin{theorem*} \label{t1}
	Let $G$ be a finite group with a coprime automorphism $\f$ that fixes $m$ elements and that satisfies the primitive ordered identity $f(x) = a_0 + a_1 \cdot x + \cdots + a_d \cdot x^d$:   
	$$g^{a_0} \cdot \f(g)^{a_1} \cdots \f^d(g)^{a_d} = 1,$$ for all $g \in G$. Then the soluble radical of $G$ has $(d,m)$-bounded index in $G$ and it has $(d,h_0)$-bounded Fitting height, where $h_0$ is the Fitting height of $C_G(\f)$.
\end{theorem*}

The bounds will be made explicit in Section \ref{s-main-theorem-proof}. In the same section, we will also show that the theorem is no longer true without the primitivity condition $\gcd(a_0,\ldots,a_d) = 1$. Whether the coprimeness condition $\gcd(|G|,|\f|) = 1$ can be removed is not even known in the ``classic case'' $f(x) = -1 + x^d$ (cf. Problem 13.8 in the Kourovka Notebook \cite{kour}). \\

The general strategy to prove the theorem is straight-forward. We begin by combining the recent results of \cite{khu-moe} with those of Turull \cite{tur0} in order to obtain an upper bound on the Fitting height of the soluble radical. This is done in Section \ref{s-soluble-case}. It then suffices to obtain an upper bound on the order of $G$ under the additional assumption that the soluble radical is the trivial group. In Section \ref{s-simple}, we do this for simple groups by means of the classification. In fact, we first prove that the automorphism has order at most $d$, and we then derive a suitable upper bound on the order of the group (without using Hartley's theorem \cite{har}). In Section \ref{s-semi-simple}, we treat the general case by reducing it to the simple case. 

\section{Definitions and examples}

An automorphism $\f$ of a finite group $G$ is \emph{coprime} if $\gcd(|G|,|\f|) = 1$. A polynomial $f(x) = a_0 + a_1 \cdot x + a_2 \cdot x^2 + \cdots + a_d \cdot x^d \in \Z[x]$ is \emph{primitive} if it is non-zero and if its \emph{content} $\gcd(a_0,a_1,a_2,\ldots,a_d)$ is $1$. The following notion was introduced by the second author.

\begin{definition}[Cf. \cite{moe1,moe2}] \label{d-ident}
	Let $f(x) = a_0 + a_1 \cdot x + a_2 \cdot x^2 + \cdots + a_d \cdot x^d \in \Z[x]$ be a polynomial. We say that $f(x)$ is an \emph{ordered identity} of the endomorphism $\gamma$ of the group $G$ if $$g^{a_0} \cdot \gamma(g)^{a_1} \cdot \gamma^2(g)^{a_2} \cdots \gamma^d(g)^d = 1,$$ for all $g \in G$. In this case, we also say that $\gamma$ \emph{satisfies} the ordered identity $f(x)$. More generally, we say that $f(x)$ is an \emph{identity} of the endomorphism $\gamma$ if there exist $b_0,b_1,b_2,\ldots,b_k \in \Z$ and $m_0,m_1,m_2,\ldots,m_k \in \Z_{\geq 0}$ such that $f(x) = b_0 \cdot x^{m_0} + b_1 \cdot x^{m_1} + b_2 \cdot x^{m_2} + \cdots + b_k \cdot x^{m_k}$ and such that $$\gamma^{m_0}(g)^{b_0} \cdot \gamma^{m_1}(g)^{b_1} \cdot \gamma^{m_2}(g)^{b_2} \cdot \gamma^{m_k}(g)^{b_k} = 1,$$ for all $g \in G$. In this case, we also say that $\gamma$ \emph{satisfies} the identity $f(x)$. 
\end{definition}

It is clear that every ordered identity of $\gamma$ is also an identity of $\gamma$, but examples show that the converse is not true in general. Some identities naturally correspond with work in the literature.

\begin{example} Let $n \in \Z_{\geq 1}$.
	\begin{enumerate}
		\item[\rm (a)] A finite group $G$ has exponent dividing $n$ if and only if the constant polynomial $f(x) := n$ is an ordered identity of some (any) automorphism. These groups have been studied extensively in the context of the restricted Burnside problem. We highlight the work of Hall---Higman \cite{ha-hi} and Zelmanov \cite{zel-rbp1,zel-rbp2}.
		\item[\rm (b)] A group is $n$-abelian if and only if it has an endomorphism $\gamma$ satisfying the linear ordered identity $f(x) := -n + x$. These groups were introduced by Baer \cite{bae} and classified by Alperin \cite{alp} (for $n > 1$).
		\item[\rm (c)] The automorphism $\f$ of the group $G$ has order dividing $n$ if and only if the polynomial $f(x) := -1 + x^n$ is an ordered identity of $\f$. Automorphisms with prescribed order have been studied extensively in the literature, as we had already observed: \cite{ber,brau-fowl,dad,gro,har,har-isa,har-mei,har-tur,hig,jab,krek-kos,pet,shu,tho59,tho64,tur0}.
		\item[\rm (d)] The automorphism $\f$ of the group $G$ is $n$-splitting if and only if the polynomial $f(x) := 1 + x + x^2 + \cdots + x^{n-1}$ is an ordered identity of $\f$. These $n$-splitting automorphisms have been studied in various contexts, including the Hughes subgroup problem and the compact Burnside problem. We mention \cite{ers,esp,hug-tho,jab94,jab96,keg,khu80,khu86,khu89,khu91,khu94,khu-mak,zel} and we refer to the references therein. We also highlight Zelmanov's powerful generalization \cite{zel17} of his solution of the compact Burnside problem to (a different kind of) polynomial identities.
	\end{enumerate} 
\end{example}

The identity in (a) is primitive if and only if $n = 1$. The identities in (b), (c), and (d) are all monic and therefore primitive. We refer to the introductions of \cite{khu-moe,moe0,moe1,moe2} for more examples and context.

\section{The soluble case} \label{s-soluble-case}

In this section, we obtain an upper bound on the Fitting height of the radical. 
We first fix some notation. For a finite soluble group $G$, the \emph{Fitting subgroup} of $G$ is denoted by $F(G)$ and the \emph{Fitting height} of $G$ is denoted by $h(G)$. The composition length of a finite cyclic group $C$ is denoted by $k(C)$, and it coincides with the number $k(|C|)$ of prime divisors of $|C|$, counted with multiplicity. \\

The following result of Thompson \cite{tho64}, Kurzweil \cite{kur}, and Turull \cite{tur0} has already been mentioned in the introduction. The sharp bound in the theorem was obtained Turull.

\begin{theorem} \label{t-tho-kur-tur}
	Let $G$ be a finite soluble group with a soluble group of operators $A$ of coprime order and of composition length $k(A)$. Then $h(G) \leq 2k(A) + h(C_G(A))$.
\end{theorem}

Our second auxiliary result is the following slight modification of Proposition 4.1 in Khukhro---Moens {\cite{khu-moe}}. We recall that, if $q$ is a prime divisor of the order of a group $G$, then $O_{q'}(G)$ is defined to be the largest normal $q'$-subgroup of $G$ and $O_{q',q}(G)$ is defined to be the inverse image of the largest normal $q$-subgroup of $G/O_{q'}(G)$.

\begin{proposition} \label{p-khu-moe}
	Let $G$ be a finite, non-trivial, soluble group and let $\f$ be a coprime automorphism that satisfies the primitive identity $f(x)$ of degree at most $d$. Let $q$ be any prime divisor of $G$ and define the quotients $\bar{G} := G/O_{q',q}(G)$ and $H := \bar{G}/F(\bar{G})$. Then the automorphism group induced by $\langle \f \rangle$ on $H$ has $d$-bounded order $| \langle \f_{|_{H}} \rangle | \leq (2d)^{(2d)}$ and $d$-bounded composition length $k(\langle |\f_{|_{H}}| \rangle) \leq 4 d$.
\end{proposition}

\begin{proof}
	We consider the analogous Proposition 4.1 in \cite{khu-moe}, which is formulated for \emph{elementary abelian identities} that \emph{do not vanish modulo any prime} divisors of $|G|$, and for automorphisms that are \emph{fixed-point-free}. We first note that $f(x)$ is an identity of the automorphism induced by $\f$ on any characteristic, elementary abelian section of $G$, so that $f(x)$ is also an elementary abelian identity of $\f$ (by definition). We next note that $f(x)$ does not vanish modulo any prime divisors of $|G|$ since $f(x)$ is assumed to be primitive. We finally note that the fixed-point-freeness of $\f$ is only used in the proof of that proposition in order to obtain the existence of Hall subgroups of $G$ that are invariant under $\f$. But it is known that, if $\sigma$ is a set of primes and if $\f$ is a coprime automorphism of a finite soluble group $G$, then $\f$ leaves some Hall $\sigma$-subgroup of $G$ invariant. So we need only replace \cite[Lemma 2.2]{khu-moe} with \cite[Remark 2.13]{jab} in the proof of \cite[Proposition 4.1]{khu-moe}.
\end{proof}

\begin{definition} \label{d-b1}
	We define $B_1 : \Z_{\geq 0} \times \Z_{\geq 0} \longrightarrow \Z_{\geq 1} : (d,m) \mapsto 8 d + m + 2$.
\end{definition}

We are now in a position to prove the first part of the main theorem.

\begin{proposition} \label{p-soluble}
	Let $G$ be a finite soluble group with an automorphism $\f$ of coprime order that satisfies a primitive identity $f(x)$ of degree at most $d$. Then $h(G) \leq B_1(d,h(C_G(\f)))$.
\end{proposition}

\begin{proof}
	We may assume that $G$ is non-trivial, since otherwise $h(G) = 0 \leq B_1(d,m)$. So we may select an arbitrary prime divisor $q$ of $|G|$ and consider the automorphism group induced by $\langle \f \rangle$ on the characteristic section $H := G/O_{q',q}(G) / F(G/O_{q',q}(G))$ of $G$. According to Theorem \ref{t-tho-kur-tur}, we have the bound $h(H) \leq 2k(\langle \f_{|_H} \rangle) +  h(C_H(\langle \f_{|_H} \rangle))$. Since $\gcd(|G|,|\f|) = 1$, the group $C_H(\langle \f_{|_H} \rangle)$ is a quotient of $C_G(\langle \f \rangle)$, so that $h(\langle C_H(\f_{|H}) \rangle) \leq h( C_G(\langle \f \rangle) )$. And, according to Proposition \ref{p-khu-moe}, we also have the bound $k(\langle \f_{|_H} \rangle) \leq 4d$. Altogether, we obtain $h(H) \leq 2(4d) + h(C_G(\f))$, and therefore $h(G/O_{q',q}(G)) \leq 8d + h(C_G(\f)) + 1$. Since $F(G) = \bigcap_{q} O_{q',q}(G)$, where $q$ runs over the prime divisors of $|G|$, we conclude that $h(G) \leq 8d + h(C_G(\f)) + 2$. 
\end{proof}

\begin{remark}
	Examples of Thompson \cite{tho64} show that the bound $h(G) \leq B_1(d,h(C_G(\f)))$ is no longer true without the coprimeness condition $\gcd(|G|,|\f|) = 1$.
\end{remark}

If $f(1) \neq 0$, then we can even obtain a bound that depends \emph{only} on $f(x)$.

\begin{corollary}
	Let $G$ be a finite soluble group with an automorphism $\f$ of coprime order that satisfies a primitive identity $f(x)$. If $f(1) \neq 0$, then $h(G) \leq 8 \deg(f(x)) + 2 {|f(1)|} + 2$.
\end{corollary}

\begin{proof}
	For every $x \in C_G(\f)$, we clearly have $1 = x^{f(1)}$, so that the exponent of the soluble group $C_G(\f)$ divides the integer $|f(1)|$. Let $|f(1)| = p_1^{e_1} \cdots p_l^{e_l}$ be the factorisation of $|f(1)|$ into distinct prime divisors $p_1,\ldots, p_l$. A result of Shalev \cite[Lemma 2.5]{sha}, based on the fundamental work of Hall---Higman \cite{ha-hi}, gives the bound $h(C_G(\f)) \leq (2e_1 +1) \cdots (2e_l +1)$ and therefore the bound $h(C_G(\f)) \leq 2 |f(1)|$. Proposition \ref{p-soluble} finishes the proof.
\end{proof}

\section{The simple case} \label{Section-simple} \label{Section-simple-strategy} \label{s-simple}

In this section, we prove the main theorem for the special case of simple groups $H$. 

\begin{lemma} \label{l-coprime}
	Let $H$ be a finite simple non-abelian group with a coprime automorphism $\f$ of order $> 1$. Then $H$ is (isomorphic to) the adjoint version of one of the untwisted groups of Lie-type \begin{equation}
			A_n(q),B_n(q),C_n(q),D_n(q),E_6(q),E_7(q),E_8(q),F_4(q),G_2(q), \label{eq-untwist}
		\end{equation} 
	or the adjoint version of one of the twisted groups of Lie-type \begin{eqnarray}
			^2 A_n(q^2),^2D_n(q^2),^2E_6(q^2),^3D_4(q^3),^2B_2(q),^2F_4(q),^2G_2(q), \label{eq-twist}
		\end{eqnarray}
		with the usual limitations on $n$ and $q$. Moreover, the automorphism $\f$ is a conjugate of a pure field automorphism. 
\end{lemma}

Here, an automorphism $\f : H  \longrightarrow H$ of $H$ is said to be a \emph{pure field automorphism} if it can be obtained by extending an automorphism $K \longrightarrow K$ of the defining field $K$ to $H$ via the root subgroups of $H$ in the usual way: cf. Definition 2.5.1 in \cite{gls} or Chapter 12 in \cite{cart} or Chapter 10 in \cite{stein}. This automorphism of $K$ is then denoted by $\f_K$.

\begin{proof}
	According to the classification, the simple group $H$ is an alternating group on at least $5$ letters, a finite simple group of adjoint Lie-type, or a sporadic group. The alternating groups do not appear in the claim since they have even order and since their outer automorphism group has exponent dividing $2$. We can similarly exclude the sporadic groups and the Tits group $^2 F_4(2)'$. This leaves only the finite simple groups of Lie-type that are listed in \eqref{eq-untwist} and \eqref{eq-twist}. Then $\operatorname{Aut}(H) $ is known to be of the form $ A \rtimes (B \times C)$, where $A$ consists of the innerdiagonal automorphisms, where $B$ consists of the graph automorphisms, and where $C$ consists of the pure field automorphisms. A case-by-case analysis shows that each non-trivial prime divisor of $|A| \cdot |B|$ also divides $|H|$. So $|\f|$ divides $|C|$, and the theorem of Schur---Zassenhaus \cite[4.7.25]{zas} conjugates $\f$ into the group $C$ of pure field automorphisms.
\end{proof}

The next lemma shows that we may assume that the automorphism $\f$ in Theorem \ref{t1} is in fact a pure field automorphism.

\begin{lemma} \label{l-conjugate}
	Let $H$ be a group with automorphism $\varphi$ of order $e$ fixing at most $m$ points and satisfying the identity $f(x)$. Let $\gamma \in \operatorname{Aut}(H)$. Then also the conjugate automorphism $\beta := \gamma \circ \varphi \circ \gamma^{-1}$ of $H$ has order $e$, fixes at most $m$ points, and satisfies the identity $f(x)$.
\end{lemma}

\begin{proof}
	Since $\beta^n = \gamma \circ \f^n \circ \gamma^{-1}$ for all $n \in \Z$, we conclude that $|\beta| = |\f|$. It is clear that $C_H(\beta) = \gamma(C_H(\f))$, so that $|C_H(\beta)| = |C_H(\f)|$. By definition, there exist $b_0,\ldots,b_k \in \Z$ and $m_0 , \ldots , m_k \in \Z_{\geq 0}$ such that $f(x) = b_0 \cdot x^{m_0} + \cdots + b_k \cdot x^{m_k}$ and such that, for all $h \in G$, we have $ \f^{m_1}(h)^{b_1} \cdot \f^{m_2}(h)^{b_2} \cdots \f^{m_k}(h)^{b_k} = 1$. For every $g \in H$, we set $h := \gamma^{-1}(g)$, and we verify that: $\beta^{m_0}(g)^{b_0} \cdot \beta^{m_1}(g)^{b_1} \cdots \beta^{m_d}(g)^{b_d} = \gamma(\f^{m_1}(h)^{b_0} \cdot \f^{m_2}(h)^{b_1} \cdots \f^{m_k}(h)^{b_k}) = \gamma(1) = 1$. 
\end{proof}

For each of the groups in Lemma \ref{l-coprime}, we now define a subgroup with good properties. 

\begin{lemma} \label{l-one-param}
	Let $H$ be one of the simple groups in \eqref{eq-untwist} or \eqref{eq-twist} with defining field $K$, other than $^2A_2(q^2)$, and let $\f$ be a pure field automorphism of $H$. Then there is an injective homomorphism $\x : (K,+) \longrightarrow H : t \mapsto \x(t) $ that satisfies $\f(\x(t)) = \x(\f_K(t))$ for all $t \in K$.
\end{lemma}

\begin{proof}
	Suppose first that $H$ is one of the untwisted groups of Lie-type \eqref{eq-untwist}. Then $H$ is generated by its root subgroups: $H = \langle \x_\alpha(t) | \alpha \in \Phi , t \in K \rangle$, where $\Phi$ is the defining root system and where $K$ is the defining field. Moreover, any such root subgroup defines an injective homomorphism $\x_\alpha : (K,+) \longrightarrow H : t \mapsto \x_\alpha(t)$ from the additive group $(K,+)$ of the field to $H$ and it satisfies $\f(\x_\alpha(t)) = \x_\alpha(\f_K(t))$, by definition. All of these claims are well-known and can be found in \cite{stein}, \cite{cart}, and \cite{gls}. \\
	
	Suppose next that $H$ is one of the twisted groups of Lie-type \eqref{eq-twist}, realized as a subgroup of an untwisted group $G$ that is fixed element-wise by a distinguished automorphism $\sigma$ (the twist). The map $t \mapsto \x(t)$ cannot be selected from the root subgroups of $G$ in this case, since the image of such a root subgroup need not be contained in $H$. But a suitable product of root subgroups, with an added twist, will naturally give rise to an injective homomorphism into $H$ that is compatible with all pure field automorphisms. \\
	
	We will define the map by using the results, terminology, and notation of Chapter 13 in Carter's book \cite{cart}. (Alternatively, the reader may consult Chapter 11 of Steinberg's lecture notes \cite{stein}, and the exact parametrizations that are given in \cite[Lemma 63]{stein}. A third reference is Chapter 2 in the book of Gorenstein---Lyons---Solomon \cite{gls}, with exact parametrizations in \cite[Table 2.4]{gls}. The reader is advised, however, to take into account subtle differences in notation, especially regarding the field parameters and structure constants.) We recall that Lemma $13.2.1$ of \cite{cart} partitions the root system $\Phi$ of $G$ into equivalence classes. The opening paragraph of \cite[13.5]{cart} lists, for each such $\Phi$, exactly which equivalence classes appear. \\
	
	Suppose first that $H$ is of type $^2A_n(q^2)$ (for some $n > 2$), or of type $^2E_6(q^2)$, or of type $^2D_n(q^2)$ (for some $n > 3$), or of type $^2F_4(q)$. Then the root system contains an equivalence class $\{r,\overline{r}\}$ of positive roots of type $A_1 \times A_1$. If the roots have the same length, then the map $\x : (K,+) \longrightarrow H : t \mapsto \x_r(t) \cdot \x_{\bar{r}}(\bar{t})$ is a well-defined, injective homomorphism according to \cite[Proposition 13.6.3 (ii)]{cart}, \cite[Proposition 13.6.4 (ii)]{cart}, and \cite[Theorem 5.3.3]{cart}. Suppose next that $r$ is a short root and $\bar{r}$ is a long root. Then the map $\x : (K,+) \longrightarrow H : t \mapsto \x_r(t^\theta) \cdot \x_{\bar{r}}({t})$ is a well-defined, injective homomorphism according to \cite[Proposition 13.6.3 (v)]{cart}, \cite[Proposition 13.6.4 (v)]{cart}, and \cite[Theorem 5.3.3]{cart}. \\ 
	
	Suppose next that $H$ is of type $^3D_4(q^3)$. In this case, the root system contains an equivalence class $\{r,\bar{r},\bar{\bar{r}}\}$ of positive roots of type $A_1 \times A_1 \times A_1$. Then the map $\x : (K,+) \longrightarrow H : t \mapsto \x_r(t) \cdot \x_{\bar{r}}(\bar{t}) \cdot \x_{\bar{\bar{r}}}(\bar{\bar{t}})$ is a well-defined, injective homomorphism, according to \cite[Proposition 13.6.3 (iii)]{cart}, \cite[Proposition 13.6.4 (iii)]{cart} and \cite[Theorem 5.3.3]{cart}. \\
	
	Suppose next that $H$ is of type $^2B_2(q)$. In this case, the root system contains an equivalence class  $\{a,b,a+b,2a+b\}$ of positive roots of type $B_2$. Then the map $\x : (K,+) \longrightarrow H : u \mapsto \x_{a+b}(u) \cdot \x_{2a+b}(u^{2 \theta})$ is a well-defined, injective homomorphism, according to \cite[Proposition 13.6.3 (vi)]{cart}, \cite[Proposition 13.6.4 (vi)]{cart}, and \cite[Theorem 5.3.3]{cart}. \\
	
	Suppose finally that $H$ is of type $^2G_2(q)$. In this case, the root system contains an equivalence class  $\{a,b,a+b,2a+b,3a+b,3a+2b\}$ of positive roots of type $G_2$. Then the map $\x : (K,+) \longrightarrow H : v \mapsto \x_{2a+b}(v^\theta) \cdot \x_{3a+2b}(v)$ is a well-defined, injective homomorphism, according to \cite[Proposition 13.6.3 (vii)]{cart},  \cite[Proposition 13.6.4 (vii)]{cart}, and \cite[Theorem 5.3.3]{cart}. \\
	
	By definition, the maps $\x_a , \x_{a+b} , \ldots $ all commute with $\f$. Moreover, all automorphisms $t \mapsto \bar{t}, v \mapsto v^\theta, \ldots$ of the defining field commute with $\f_K$. So each of the above maps $\x$ commutes with $\f$.
\end{proof}

\begin{remark}
	We have avoided using the equivalence classes of positive roots of type $A_1$ because the corresponding (well-defined, injective) homomorphism is defined only on a \emph{proper subfield} of the defining field. This will not be enough to prove Proposition \ref{p-order}. 
\end{remark}

For technical reasons, we treat the projective special unitary groups $^2A_2(q^2)$ separately. 

\begin{lemma} \label{l-def-domain}
	Let $K$ be a field of $q^2$ elements and let $N$  be an integer. Then the subset $$D_{N,K} := \{ (s,u) \in K \times K |u + {u}^q = - N \cdot s \cdot {s}^q \}$$ of $K \times K$ is a group with respect to the operation $\ast : (K \times K) \times (K \times K) \longrightarrow K \times K : ((s,u),(t,v)) \mapsto (s+t,u+v - N \cdot s^q \cdot t)$, the inversion $\iota: K \longrightarrow K : (s,u) \mapsto (-s,u^q)$, and the neutral element $(0,0)$. The projection of $D_{N,K}$ onto its first coordinate is all of $K$. If $\f_K$ is an automorphism of $K$ and if $(s,u) \in D_{N,K}$, then also $(\f_K(s),\f_K(u)) \in D_{N,K}$.
\end{lemma}

\begin{proof}
	The first and third claim can be verified by a routine computation. For the second claim, we consider the unique subfield $L$ of $K$ that consists of exactly $q$ elements and we consider the map $K \longrightarrow K : \beta \mapsto \beta + \beta^q$. Then the kernel of this map is $L$ and the image of the map is contained in $L$. So there exist $\alpha \in L \setminus \{0\}$ and $\beta \in K \setminus L$ such that $(\beta + \beta^q)/\alpha = 1$. Now let $s \in K$ be arbitrary. Then clearly $s \cdot s^q \in L$, and we define $u := - N \cdot s \cdot s^q \cdot \beta / \alpha$. Since $u + {u}^q = - N \cdot s \cdot s^q \cdot \beta / \alpha - N \cdot s^q \cdot s \cdot \beta^q / \alpha = - N \cdot s \cdot {s}^q$, we have $(s,u) \in D_{N,K}$.
\end{proof}

One can verify that $(D_{N,K},\ast)$ is a special group of order $q^3$. 

\begin{lemma} \label{l-hom-psu22}
	Let $H$ be a finite simple adjoint group of type $^2A_2(q^2)$ with defining field $K$ and let $\f$ be a pure field automorphism of $H$. Then there is an integer $N$ and an injective homomorphism $\x : (D_{N,K},\ast) \longrightarrow H$ such that $\f(\x(s,u)) = \x(\f_K(s),\f_K(u))$ for all $(s,u) \in D_{N,K}$.
\end{lemma}

\begin{proof}
	As in the proof of Lemma \ref{l-one-param}, we obtain an equivalence class $\{r,\bar{r}\}$ of positive roots of type $A_2$ with corresponding structure constant $N_{r,\bar{r}}$ (still in the notation of \cite{cart}). Set $N := N_{r,\bar{r}}$ and consider the group $(D_{N,K},\ast)$ that was defined in Lemma \ref{l-def-domain}. Then the map $\x : (D_{N,K},\ast) \longrightarrow {^2A_2}(q^2) : (s,u) \mapsto \x_r(s) \cdot \x_{\bar{r}}(\bar{s}) \cdot \x_{r + \bar{r}}(u)$ is a well-defined, injective homomorphism, according to \cite[Proposition 13.6.3 (iv)]{cart}, \cite[Proposition 13.6.4 (iv)]{cart}, and \cite[Theorem 5.3.3]{cart}. Since $\f$ commutes with the $\x_r, \x_{\bar{r}}, \x_{r + \bar{r}}$, and since $\f_K$ commutes with the automorphism $s \mapsto \bar s$ of the field, we once more obtain $\f(\x(s,u)) = \x(\f_K(s),\f_K(u))$ for all $(s,u) \in D_{N,K}$.
\end{proof}

We now use these subgroups to pull the identity $f(x)$ of $\f$ down to the level of the defining field. This allows us to bound the order of $\f$.

\begin{proposition} \label{p-order}
	Let $H$ be one of the groups of type \eqref{eq-untwist} or \eqref{eq-twist} and let $\f$ be a pure field automorphism of $H$ satisfying the primitive identity $f(x)$. Then $|\f| \leq \deg(f(x))$.
\end{proposition}

\begin{proof}
	As before, we consider $H$ to be a subgroup of an untwisted group $G$ with (possibly trivial) twist $\sigma$. Let the pure field automorphism $\f : H \longrightarrow H$ be induced by the automorphism $\f_K : K \longrightarrow K : t \mapsto t^{q_0}$ of the defining field $K$. Let the identity $f(x)$ be given by $a_0 + a_1 \cdot x + \cdots + a_d \cdot x^d$, with $a_d \neq 0$. By definition, there exist $b_0,\ldots,b_k \in \Z$ and $m_0 , \ldots , m_k \in \Z_{\geq 0}$ such that $f(x) = b_0 \cdot x^{m_0} + \cdots + b_k \cdot x^{m_k}$ and such that, for all $h \in H$, we have \begin{equation}
		\f^{m_0}(h)^{b_0} \cdot \f^{m_1}(h)^{b_1} \cdots \f^{m_k}(h)^{b_k} = 1. \label{eq-ident}
	\end{equation}
	Suppose first that $H$ is of type $^2A_2(q^2)$. Let $\x : (D_{N,K},\ast) \longrightarrow H$ be the injective homomorphism of Lemma \ref{l-hom-psu22} and select an arbitrary $s \in K$. According to Lemma \ref{l-def-domain}, there exists some $u \in K$ such that $(s,u) \in D_{N,K}$. By evaluating \eqref{eq-ident} in $h := \x(s,u)$, we obtain 
	\begin{eqnarray*}
		1 &=& \f^{m_0}(\x(s,u))^{b_0} \cdot \f^{m_1}(\x(s,u))^{b_1} \cdots \f^{m_k}(\x(s,u))^{b_k} \\ \label{eq-comp1}
		&=& \x( (\f_K^{m_0}(s),\f_K^{m_0}(u))^{b_0} \ast (\f_K^{m_1}(s),\f_K^{m_1}(u))^{b_1} \ast \cdots \ast (\f_K^{m_k}(s),\f_K^{m_k}(u))^{b_k}  ) \\
		&=& \x( b_0 \cdot \f_K^{m_0}(s) + b_1 \cdot \f_K^{m_1}(s) + \cdots + b_k \cdot \f_K^{m_k}(s),v) \\
		&=& \x( a_0 \cdot s + a_1 \cdot \f_K(s) + \cdots + a_d \cdot \f_K^d(s),v) , \label{eq-comp4}
	\end{eqnarray*}
	for some $v \in K$ such that $(a_0 \cdot s + a_1 \cdot \f_K(s) + \cdots + a_d \cdot \f_K^d(s),v) \in D_{N,K}$. Since the map $\x$ is injective, we have $0_{K} = a_0 \cdot s + a_1 \cdot \f_K(s) + \cdots + a_d \cdot \f_K^d(s)$, and therefore \begin{equation}
		0_{K} = a_0 \cdot s^{{q_0}^0 - 1} + a_1 \cdot s^{{q_0}^1 -1} + \cdots + a_d \cdot s^{{q_0}^d - 1}, \label{eq-field-ident}
	\end{equation} for all $s \in K^\times$. Let $\omega$ be a generator of the cyclic group $(K^\times,\cdot)$ of the finite field $K$. By evaluating \eqref{eq-field-ident} in the iterated powers $\omega^0 , \omega^1 , \ldots , \omega^d$ of $\omega$, we see that the vector $(a_0,\ldots,a_d)^T \in K^{d+1}$ is a solution of the homogeneous Vandermonde system
	$$
	\left( 
	\begin{array}{ccccc}
		\omega^{0({q_0}^0-1)}	& \omega^{0({q_0}^1-1)} & \omega^{0({q_0}^2-1)} & \cdots & \omega^{0({q_0}^d-1)} \\
		\omega^{1({q_0}^0-1)}	& \omega^{1({q_0}^1-1)} & \omega^{1({q_0}^2-1)} & \cdots & \omega^{1({q_0}^d-1)} \\
		\vdots	& \vdots & \vdots & \ddots & \vdots \\
		\omega^{(d-1)({q_0}^0-1)}	& \omega^{(d-1)({q_0}^1-1)} & s^{(d-1)({q_0}^2-1)} & \cdots & \omega^{(d-1)({q_0}^d-1)} \\
		\omega^{d({q_0}^0-1)}	& \omega^{d({q_0}^1-1)} & s^{d({q_0}^2-1)} & \cdots & \omega^{d({q_0}^d-1)} 
	\end{array}
	\right) \cdot
	\left( 
	\begin{array}{c}
		a_0	\\
		a_1	\\
		\vdots	\\
		a_{d-1}	\\
		a_d	
	\end{array}
	\right) 
	=
	\left( 
	\begin{array}{c}
		0	\\
		0	\\
		\vdots	\\
		0	\\
		0	
	\end{array}
	\right).
	$$
	The determinant $\pm \prod_{0 \leq i < j \leq d} (\omega^{{q_0}^i - 1} - \omega^{{q_0}^j - 1})$ of this matrix vanishes in $K$, since otherwise the coefficients $a_0,\ldots,a_d$ of $f(x)$ all vanish modulo $p$, which would contradict the fact that $f(x)$ is primitive. So there exist $0 \leq i < j \leq d$ such that $\omega^{{q_0}^i - 1} = \omega^{{q_0}^j - 1}$ and therefore $\f_K^{i}(t) = \f_K^{j}(t)$ for all $t \in K$. We see, in particular, that the order of $\f_K$ is at most $d$. So the order of $\f$ on the generating set $\{ \x_\alpha(t) | \alpha \in \Phi , t \in K\}$ of $G$ is at most $d$. The order of $\f$ on the subgroup $H$ of $G$ is therefore also at most $d$. \\
	
	Suppose next that $H$ is of type \eqref{eq-untwist} or \eqref{eq-twist}, but \emph{not} of type $^2A_2(q^2)$. Let $\x : (K,+) \longrightarrow H$ be the injective homomorphism of Lemma \ref{l-one-param}. We then select an arbitrary $s \in K$ and evaluate \eqref{eq-ident} in $h := \x(s)$ in order to obtain
	\begin{eqnarray*}
		1 &=& \f^{m_0}(\x(s))^{b_0} \cdot \f^{m_1}(\x(s))^{b_1} \cdots \f^{m_k}(\x(s))^{b_k} \\
		  &=& \x( b_0 \cdot \f_K^{m_0}(s) + b_1 \cdot \f_K^{m_1}(s) + \cdots + b_k \cdot \f_K^{m_k}(s) ) \\
		  &=& \x( a_0 \cdot s + a_1 \cdot \f_K(s) + \cdots + a_d \cdot \f_K^d(s) ).
	\end{eqnarray*}
	Since the map $\x$ is injective, we have $0_K = a_0 \cdot s + a_1 \cdot \f_K(s) + \cdots + a_d \cdot \f_K^d(s)$, for all $s \in K$. The rest of the proof can now be repeated verbatim.
\end{proof}

\begin{definition} \label{d-b3}
	We define $B_3 : \Z_{\geq 0} \times \Z_{\geq 1} \longrightarrow \Z_{\geq 1} : (d,m) \mapsto m + m^{1000 \cdot d} $.
\end{definition}

The following lemma gives us an explicit upper bound on the order of our simple group $H$ in terms of $|\f|$ and $|C_H(\f)|$. We have made no attempt to make this bound optimal.

\begin{lemma} \label{l-Lie-type-bd}
	Let $H$ be the adjoint version of a finite group of Lie-type. Suppose that $H$ has a pure field automorphism $\f$ of coprime order. Then $|H| \leq B_3(|\f|,|C_H(\f)|)$
\end{lemma}

\begin{proof}
	Let $n$ be the (untwisted) rank of $H$ and let $K$ be the defining field. We first claim that $C_H(\f)$ contains a subgroup $S$ that is isomorphic to a group of exactly the same adjoint type as $H$, but that is defined over a subfield $K_0$ of $K$ such that $|K| \leq |K_0|^{|\f|}$. Let us prove this claim by induction on the number $k = k(|\f|)$ of prime divisors of $|\f|$, counted with multiplicity. The base of the induction corresponds with $k = 0$, in which case $C_H(\f) = H$, so that there is nothing to prove. So we assume $k > 0$ and we let $p$ be a prime divisor of $|\f|$. According to \cite[Proposition 4.9.1 (a) and (b)]{gls}, $M := O^{p'}(C_H(\f^{|\f|/p}))$ is a group of the same Lie-type as $H$ and it is also adjoint, but $M$ is defined over the subfield $K_1$ of $K$ satisfying $|K| = |K_1|^p$. By construction, $\f(M) = M$ and the order of $\f$ on $M$ divides $|f|/p$. Since $\gcd(|M|,|\f_{|_M}|) = 1$, Lemma \ref{l-coprime} allows us to conclude that $\f$ is again a field automorphism of $M$. According to Lemma \ref{l-conjugate}, it now suffices to prove the claim under the additional assumption that $\f_{|_M}$ is a pure field automorphism of $M$. Induction then gives us a subgroup $S$ of $C_M(\f_{|_M}) \subseteq C_H(\f)$ of exactly the correct type over a subfield $K_0$ of $K_1$ satisfying $|K_1| \leq |K_0|^{|\f_{|_{M}}|} \leq |K_0|^{|\f|/p}$ and therefore $|K| \leq |K_0|^{|\f|}$. This establishes the claim. \\
	
	Some coarse estimates for adjoint groups of Lie-type will finish the proof. Since $S$ has rank $n$ and defining field $K_0$, we have {the (coarse) bound $|K_0|^{n^2} \leq |S|^4$} and therefore the bound $|K|^{n^2} \leq |K_0|^{n^2 \cdot |\f|} \leq |S|^{4 \cdot |\f|} \leq |C_H(\f)|^{4 \cdot |\f|}$. Since $H$ has rank $n$ and defining field $K$, we have {the (coarse) bound $|H| \leq |K|^{248 \cdot n^2}$} and therefore $|H| \leq |C_H(\f)|^{992 \cdot |\f|}$.
\end{proof}

\begin{remark}
	An upper bound on $|G|$ that does not require $|G|$ and $|\f|$ to be coprime can be found in Hartley's generalization {\cite[Theorem~A']{har}} of the classic Brauer---Fowler theorem \cite{brau-fowl}. That more generally applicable bound is, however, not explicit and it appears to grow rather quickly. On the other hand, particularly good bounds for involutions have recently been obtained by Guralnick---Robinson \cite{gur-rob}. Bounds using only elementary methods have also recently been obtained by Jabara \cite{jab21}.
\end{remark}

We can finally prove the main theorem for simple groups.

\begin{proposition} \label{c-simple}\label{p-simple}
	Let $H$ be a finite simple non-abelian group with a coprime automorphism $\f$, fixing at most $m$ points and satisfying a primitive identity of degree at most $d$. Then $|H| \leq B_3(d,m)$.
\end{proposition}

\begin{proof}
	We may assume that $|\f| > 1$, since otherwise $|H| = |C_H(\f)| \leq m \leq B_3(d,m)$. According to Lemma \ref{l-coprime}, $H$ can then be found in the list \eqref{eq-untwist} or \eqref{eq-twist}, and $\f$ is a conjugate of a pure field automorphism. Lemma \ref{l-conjugate} therefore allows us to assume that $\f$ is a pure field automorphism. According to Proposition \ref{p-order}, this automorphism has order at most $d$. Lemma \ref{l-Lie-type-bd} therefore allows us to conclude that $|H| \leq B_3(|\f|,|C_H(\f)|) \leq B_3(d,m)$. 
\end{proof}

\section{The ``semi-simple'' case} \label{s-semi-simple}

We now consider the main theorem for groups modulo their soluble radical, i.e. the ``semi-simple'' case. We begin by recalling a well-known consequence of the classification of the finite simple groups.

\begin{theorem}[Rowley \cite{row}] \label{t-row}
	Let $G$ be a finite group with a fixed-point-free automorphism. Then $G$ is soluble.
\end{theorem}

We introduce minor variations on the auxiliary polynomials of previous papers.

\begin{definition}[Cf. \cite{moe0,moe1,moe2}]
	Let $f(x) = a_0 + a_1 \cdot x + \cdots + a_d \cdot x^d \in \Z[x]$. For each positive integer $n$ and each nonnegative integer $j\leq n-1$, we define the partial sums $$f_{n,j}(x^n) \cdot x^j := \sum_{i \equiv j \operatorname{mod} n} a_i \cdot x^i,$$ so that $f(x) = f_{n,0}(x^n) + f_{n,1}(x^n) \cdot x + \cdots + f_{n,n-1}(x^n) \cdot x^{n-1}$.
\end{definition}

\begin{remark} \label{r-coprime}
	Let $f(x) \in \Z[x] \setminus \{0\}$ and $n \geq 1$. If $f(x)$ is primitive, then there is some $0 \leq j \leq n-1$ such that also $f_{n,j}(x)$ is primitive.
\end{remark} 

\begin{definition} \label{d-b2}
	We recursively define 
	$$B_2 : \Z_{\geq 0} \times \Z_{\geq 1} \longrightarrow \Z_{\geq 1} : (d,m) \mapsto 
	\begin{cases}
		1 & \text{ if } m = 1, \\
		B_3(d,m)^{m \cdot d} ! \cdot B_2(d, \lfloor m/2 \rfloor) & \text{ if } m > 1.
	\end{cases}
	$$
\end{definition}

We are now in a position to prove the second claim of our main theorem.

\begin{proposition} \label{p-index}
	Let $G$ be a finite group with a coprime automorphism $\f$ fixing at most $m$ points and satisfying the primitive \emph{ordered} identity $f(x)$. Then $|G/R(G)| \leq B_2(\deg(f(x)),m)$.
\end{proposition}

We consider Hartley's generalized Brauer---Fowler theorem \cite[Theorem~A]{har} and we generalize its proof in a straightforward way.

\begin{proof} 
	Suppose first that there is a simple, non-abelian group $H$ and a non-negative integer $k$ such that $G = H_1 \times \cdots \times H_k$ and such that each $H_i$ is isomorphic to $H$. The automorphism induced on $G$ permutes these factors $H_i$, so that $G$ breaks up into orbits $T_1, \ldots , T_l$. On each such orbit, the induced automorphism must have a non-trivial fixed-point by Theorem \ref{t-row}. So the number of orbits $l$ satisfies $l \leq |C_G(\f)| \leq m$. \\
	
	Now consider any one of these orbits, say $T_i$, and let $n_i$ be the number of simple factors of $T_i$. Then $n_i \geq 1$ and each factor is isomorphic to $H$. After re-labeling, we may assume that $T_i = H_0 \times \cdots \times H_{n_i - 1}$ and that $\f$ cyclically permutes these simple factors: $\f(H_j) = H_{j \operatorname{mod} l_i}$. Let $f(x)$ be given by $a_0 + a_1 \cdot x + \cdots + a_d \cdot x^d$, with $a_d \neq 0$. Suppose first that $d \leq n_i -1$. For each $g \in H_1$, we have by assumption the equality $1 = g^{a_0} \cdot \f(g^{a_1}) \cdots \f^d(g^{a_d})$. Since the elements $g^{a_0} , \f(g)^{a_1} , \ldots , \f^d(g)^{a_d}$ belong to different $H_j$, we have $g^{a_0} = \cdots = g^{a_d} = 1$, so that $g = g^{\operatorname{cont}(f(x))} = 1$. We conclude that $H = \{1\}$. This contradiction shows that each orbit $T_i$ has at most $d$ simple factors: $n_i \leq d$. We conclude therefore that the number $k$ of simple factors of $G$ satisfies $k = n_1 + \cdots + n_l \leq m \cdot d$. \\
	
	As in the previous paragraph, we observe that $\f^j(g)^{a_j} \in H_{j \operatorname{mod} n_i}$, for all $g \in H_0$ and all $j \in \{0,1,\ldots,n_i - 1 \}$. So we may conclude that each $f_{n_i,j}(x)$ is an ordered identity of the automorphism induced by $\f^{n_i}$ on $H_0$. According to Remark \ref{r-coprime}, we can select some $j \in \{0,1,\ldots,n_i -1\}$ such that $f_{n_i,j}(x)$ is primitive. One can further verify that $|C_{H_0}(\f_{|_{H_0}}^{n_i})| \leq |C_G(\f)| \leq m$. So the assumptions of Proposition \ref{p-simple} are satisfied for the simple group $H_0$, the automorphism $\f_{|_{H_0}}^{n_i}$, and the identity $f_{n_i,j}(x)$. We may therefore conclude that $|H| = |H_0| \leq B_3(d,m)$. Altogether, we obtain the bound $|G| = |H|^k \leq B_3(d,m)^{m \cdot d} =: B$, and therefore the coarse bound $|\operatorname{Aut}(G)| \leq B!$. \\
	
	We finally consider the general case, which we prove by induction on $m = |C_G(\f)|$. By passing from $G$ to $G/R(G)$, we may assume that $R(G) = \{1\}$. Define $N:= \bigcap_{S} C_G(S)$, where $S$ runs over the characteristic, characteristically-simple, non-abelian sections of $G$. Then $N$ is a normal soluble subgroup of $G$ and therefore the trivial group. By the above, every characteristic section $S$ of $G$ that is characteristically-simple but not abelian satisfies $|G/C_G(S)| \leq B!$. So we obtain a family $\Lambda$ of normal, $\f$-invariant subgroups of $G$ of index at most $B!$ and with trivial intersection. If $m = 1$, then Theorem \ref{t-row} implies $|G| = 1 \leq B_2(d,m)$. So we may assume that $m > 1$. Then there is some $L \in \Lambda$ that does not contain all the fixed-points of $\f$ in $G$. In this case, we have $|C_L(\f)| \leq \lfloor m/2 \rfloor$, so that we may apply the induction hypothesis to $L$, $\f_{|_L}$, and $f(x)$ in order to obtain $|L/R(L)| \leq B_2(d,\lfloor m/2 \rfloor)$. So we have $[G:R(L)] \leq [G:L] \cdot [L:R(L)] \leq B! \cdot B_2(d, \lfloor m/2 \rfloor ) = B_2(d,m)$. Since the soluble radical $R(L)$ of $L$ is characteristic in $L$, it is a soluble normal subgroup of $G$, and therefore trivial. So we may indeed conclude that $|G| = [G:R(L)] \leq B_2(d,m)$.
\end{proof}

\begin{remark} \label{r-ordered-nordered}
	This proof uses the fact that $f(x)$ is an \emph{ordered} identity of the automorphism. The proof can also be made to work (in the obvious way) for identities $\{ \f^{m_0}(g)^{b_0} \cdots \f^{m_k}(g)^{b_k} | g \in G \} = \{1\}$ that are not necessarily ordered, at the cost of replacing the invariant $\deg(f(x)) = \deg( b_0 \cdot x^{m_0} + \cdots + b_k \cdot x^{m_k} )$ with the possibly larger invariant $\max \{m_0 , \ldots , m_k \}$.
\end{remark}

\section{Proof of the main theorem} \label{s-main-theorem-proof}

We can now prove the main theorem with the bounds of Definitions \ref{d-b1} and \ref{d-b2}.

\begin{theorem} 
	Let $G$ be a finite group with a coprime automorphism $\f$ fixing $m$ elements and satisfying an ordered identity that is primitive and of degree at most $d$. Then the soluble radical $R$ of $G$ satisfies $$h(R) \leq B_1(d,h(C_G(\f))) \leq B_1(d,m) \text{ and } |G/R| \leq B_2(d,m).$$
\end{theorem}

\begin{proof}
	Let $\f_{|_R}$ be the restriction of $\f$ to $R$. Then $\f_{|_R}$ satisfies the same ordered identity and $C_R(\f_{|_R}) = C_G(\f) \cap R \subseteq C_G(\f)$. We may therefore apply Proposition \ref{p-soluble} in order to obtain the bound $h(R) \leq B_1(d,h(C_R(\f_{|_R}))) \leq B_1(d,h(C_G(\f))) \leq B_1(d,m)$. According to Proposition \ref{p-index}, we also have the bound $|G/R| \leq B_2(d,m)$.
\end{proof}

In view of Remark \ref{r-ordered-nordered}, we also obtain the following analogue of the main theorem for identities that are not necessarily ordered. Let $m_0,\ldots,m_k \in \Z_{\geq 0}$, let $b_0,\ldots,b_k \in \Z$, and define the polynomial $f(x) := b_0 \cdot x^{m_0} + b_1 \cdot x^{m_1} + \cdots + b_k \cdot x^{m_k} \in \Z[x]$.

\begin{theorem} \label{t1b}
	Let $G$ be a finite group with a coprime automorphism $\f$ fixing $m$ elements and satisfying the primitive identity $$ \f^{m_0}(g)^{b_0} \cdot \f^{m_1}(g)^{b_1} \cdots \f^{m_k}(g)^{b_k} = 1,$$ for all $g \in G$. Then the soluble radical $R$ of $G$ satisfies $h(R) \leq B_1(d,h(C_G(\f))) \leq B_1(d,m)$ and $|G/R| \leq B_2(d,m)$, where $d:= \max \{m_0, \ldots , m_k\}$. 
\end{theorem}

Both theorems are generally false for polynomials with non-trivial content. 

\begin{example} \label{ex-prim}
	Let $S$ be a finite simple non-abelian group and let $n \in \Z_{\geq 1}$. Then each $G_n := S \times \cdots \times S$ (with $n$ factors) admits an automorphism $\f : G_n \longrightarrow G_n : (g_1,g_2,\ldots,g_n) \mapsto (g_n,g_1,\ldots,g_{n-1})$ that fixes exactly $|S|$ elements and that satisfies the constant ordered identity $f(x) := |S|$. But $\lim_{n \rightarrow + \infty} |G_n /R(G_n)| = \lim_{n \rightarrow + \infty} |S|^n = + \infty$. 
\end{example}

\section*{Acknowledgements}

The research was supported by the Austrian Science Fund (FWF) Projects: P 30842--N35 and I 3248--N35. The author would also like to thank E. Khukhro for his feedback on an early draft and R. Lyons for his help with the references in Section \ref{s-simple}.

\end{document}